\documentclass[12pt,fleqn]{article}
\usepackage{amsfonts,amssymb,amsmath,amsthm,bbm,xcolor,ulem}
\newtheorem{theorem}{Theorem}[section]
\newtheorem{lemma}[theorem]{Lemma}

\newtheorem{proposition}[theorem]{Proposition}	
\newtheorem{corollary}[theorem]{Corollary}

\newtheorem{remark}[theorem]{Remark}
\newtheorem{conjecture}[theorem]{Conjecture}
\newtheorem{Note}[theorem]{Notation}
\theoremstyle{definition}
\numberwithin{equation}{section}

\def\a{\mathfrak{a}}

\def\R{{\bf R}}

\def\C{{\bf C}}
\def\F{{\bf F}}
\def\H{{\bf H}}

\def\to{\rightarrow}

\begin{document}

\allowdisplaybreaks

\begin{center}
Sharp estimates for $W$-invariant Dunkl and heat kernels in the $A_n$ case\\
Piotr Graczyk and Patrice Sawyer
\end{center}

\begin{abstract}
In this article, we prove exact estimates for the 
$W$-invariant Dunkl kernel and   heat kernel, for the root system of type $A$ with arbitrary positive multiplicities. We apply the estimates of the  $W$-invariant Dunkl  heat kernel
to compute sharp estimates for the Newton kernel and for the $s$-stable semigroups generated by a fractional power of the $W$-invariant Dunkl Laplacian.
\end{abstract}




\allowdisplaybreaks

\section{Introduction and notations}\label{intro}

	In Dunkl analysis for a root system $\Sigma$ on $\R^d$, a crucial role is played by the Dunkl kernel $E_k(X,Y)$ and by the Dunkl heat kernel  $p_t(X,Y)$.
	Finding good estimates of  the kernels $E_k$ and of the
Dunkl heat kernel $p_t$ 
is a challenging and important subject, developed recently in \cite{AnkerDH,PGPS1}. In this paper we prove exact estimates of both these kernels in the $W$-radial rational Dunkl case, for the root system $A_n$ with arbitrary positive multiplicities. 
	
	For a good introduction on rational Dunkl theory, the reader should consider the paper \cite{AnkerDH} or the book \cite{PGMRMY}.
	We provide here some details and notations on Dunkl analysis.
	
	For every root $\alpha\in \Sigma$, let $\sigma_\alpha(X)=X-2\,\frac{\langle \alpha,X\rangle}{\langle \alpha,\alpha\rangle}\,\alpha$. 
	The Weyl group $W$ associated to the root system is generated by the reflection maps $\sigma_\alpha$.
	
	A function $k: \Sigma \to \R$
	is called a multiplicity function if it is invariant under the action of $W$
	on $\Sigma.$
	
Let $\partial_\xi$ be the derivative in the direction of $\xi\in\R^d$.  
	The Dunkl operators indexed by $\xi$ are then given by
	\begin{align*}
	T_\xi(k)\,f(X)&=\partial_\xi\,f(X)+\sum_{\alpha\in \Sigma_+}\,k(\alpha)\,\alpha(\xi)\,\frac{f(X)-f(\sigma_\alpha\,X)}{\langle \alpha,X\rangle}.
	\end{align*}
	The $T_\xi$'s, $\xi\in\R^d$, form a commutative family.

	For fixed $ Y\in\R^d$, the Dunkl kernel $E_k(\cdot,\cdot)$ is then the only real-analytic solution to the system
	\begin{align*}
\left.T_\xi(k)\right\vert_X\,E_k(X,Y)
	=\langle \xi,Y\rangle\, E_k(X,Y),~\forall\xi\in\R^d
	\end{align*}
	with $E_k(0,Y)=1$.  In fact, $E_k$ extends to a holomorphic function  on $\C^d\times \C^d$.

	Its $W$-invariant version $E^W_k(X,\lambda)$ is called  a Bessel function of Dunkl type (see \cite[p.{} 57]{PGMRMY}) or a spherical function $\psi_\lambda(X)$ of type $\Sigma$ (refer to \cite{Sawyer}).
	In this paper we use the latter terminology and notation.  We have
	\begin{align*}
	\psi_\lambda(X) =E^W_k(X,\lambda)=\frac{1}{|W|}\sum_{w\in w}\,E_k(w\cdot X,\lambda)
	\end{align*}
	and $\psi_\lambda(X)$ is the only real-analytic solution of the system
	\begin{align*}
	\left.p(T_{{\bf e}_1},\dots,T_{{\bf e}_d})(k)\right\vert_X\,
	\psi_\lambda(X)
	=p(\lambda)\,\psi_\lambda(X),
		\qquad
	~\forall \lambda \in\R^d
	\end{align*}
	for every Weyl-invariant polynomial $p$ (here ${\bf e}_1$, \dots, ${\bf e}_d$ represent the standard basis on $\R^d$).

Let $\omega_k(X):=
\prod_{\alpha\in \Sigma^+}\,\vert\langle \alpha,X\rangle\vert^{2\,k(\alpha)}
$
be the Dunkl weight function  on $\R^d$.
Recall that the 
Dunkl transform
of a $W$-invariant function $f$
on $\R^d$
\begin{align*}
\hat f(\lambda):= c_k^{-1}
\int f(x) \psi_{-i\lambda}(X)
\omega_k(X)dX,\qquad \lambda\in \R^d,
\end{align*}
plays the role of the spherical Fourier transform in $W$-invariant Dunkl analysis.	
(here the constant  $c_k$ is the Macdonald--Mehta--Selberg integral.)

In \cite[Conjecture 18]{PGPS1}, we made the following conjecture on the growth of  spherical functions $\psi_\lambda(x)$
of type $\Sigma$. 
We use the  Cartan algebra notation $\a=\R^d$. Then $\a^+$
denotes the open positive Weyl chamber with respect to a system $\Sigma^+$  of positive roots.

\begin{conjecture}\label{C}
If $\lambda$, $X\in\overline{\a^+}$, then
\begin{align*}
\psi_\lambda(e^X)
\asymp\frac{e^{\lambda(X)}}{\prod_{\alpha>0}\,(1+\alpha(X)\,\alpha(\lambda))^{k(\alpha)}}.
\end{align*}

For the root system $A_n$ on $\R^d$, $d \geq n$, and multiplicity $k(\alpha)=k>0$, this becomes
\begin{align}
\psi_\lambda(e^X)
\asymp\frac{e^{\lambda(X)}}{\prod_{i<j\leq n+1}\,(1+(x_i-x_j)\,(\lambda_i-\lambda_j))^k},
\qquad
\lambda, X\,\in\overline{\a^+}
\label{CC}
\end{align}
(the underlying constants here only depend on $k$).
\end{conjecture}

The notation $f\asymp g$ in a domain $D$ means that there exists $C_1>0$ and $C_2>0$ such that
$C_1\,g(x)\leq f(x)\leq C_2\,g(x)$ with $C_1$ and $C_2$ independent of $x\in D$. 
Recall that for the root system $A_n$ on $\R^{n+1}$, the positive Weyl chamber is defined by 
$\a^+=\{X\in \R^{n+1}\,|\, x_1>x_2> \ldots >x_{n+1}$\}.

\begin{remark}
This conjecture includes the cases of the symmetric spaces of noncompact type ${\bf POS}_1(n,\F)$,
the positive definite matrices of determinant 1 over $\F$ where $\F=\R$ (the real numbers with $k=1/2$), $\F=\C$ (the complex numbers with $k=1$), $\F=\H$ (the quaternion numbers with $k=2$ or $\F={\bf O}$ (the Octonions with $k=4$) when $n=3$.
In \cite{PGPS1}, we proved the conjecture for the root system $A_n$ in the complex case $k=1$.
\end{remark}

The main tool of the proof
of Conjecture \ref{C} for root systems $A_n$ is
the following iterative formula for the spherical functions of type $A$,
proven in \cite{Sawyer}.
Here we do not assume that the elements of the Lie algebra have trace 0. Here the Cartan subalgebra $\a$ for the root system $A_n$ is isomorphic to $\R^{n+1}$. For $\lambda\in\a =\R^{n+1}$
and $X\in \a^+$,we have
\begin{align}
\psi_\lambda(e^X)&=e^{\lambda(X)}\ \hbox{if $n=1$ and}\nonumber\\
\psi_\lambda(e^X)&=\frac{\Gamma(k\,(n+1))}{(\Gamma(k))^{n+1}}\,e^{\lambda_{n+1}\,\sum_{r=1}^{n+1}\,x_r}\,\pi(X)^{1-2\,k}\,\int_{x_{n+1}}^{x_n}\,\cdots\,\int_{x_2}^{x_1}
\,\psi_{\lambda_0}(e^Y)\label{iter}
\\&\left[\prod_{i=1}^n\,\left(\prod_{j=1}^i\,(x_j-y_i)\,\prod_{j=i+1}^{n+1}\,(y_i-x_j)\right)\right]^{k-1}
\,\prod_{i<j\leq n}\,(y_i-y_j)\,dy_1\cdots dy_n\nonumber
\end{align}
where $\lambda_0(U)=\sum_{r=1}^n\,(\lambda_r-\lambda_{n+1})\,u_k$
and $\pi(X)= \prod_{i<j\leq n+1}\,(x_i-x_j) $.\\

\begin{remark}
Formula \eqref{iter} concerns the action of the root system $A_{n}$ on $\R^{n+1}$.  If we assume $\sum_{k=1}^{n+1}\,x_k=0=\sum_{k=1}^{n+1}\,\lambda_k$, we have then the action of the root system $A_{n}$ on $\R^{n}$.  We can also consider the action of $A_{n}$ on any $\R^m$ with $m\geq n$ by  deciding on which $n+1$ entries $x_k$, the roots  act.  These considerations do not affect the results of this article.
\end{remark}

The Dunkl heat kernel $p_t(X,Y)$ is given as
\begin{align} 
p_t(X,Y)
&=\frac1{2^{\gamma+d/2} c_k}\,t^{-\frac{d}{2}-\gamma}\,e^{\frac{-|X|^2-|Y|^2}{4t}}\,E_k\left(X,\frac{Y}{2t}\right),\label{heatSpher}
\end{align}
where $\gamma= \sum_{\alpha>0} k(\alpha)$.
Establishing estimates of the Dunkl heat kernel is equivalent to estimating the Dunkl kernel as demonstrated by equation \eqref{heatSpher}.

In \cite[Lemma 4.5]{Roesler}, it is shown that
\begin{align*}
&\int_{\R^d}\,p_t(X,Y)\,\omega_k(Y)\,dY=1\\
&\left.\Delta_k\right\vert_X\,p_t(X,Y)
=\frac{\partial~}{\partial t}\,p_t(X,Y)
\end{align*}
where the Dunkl Laplacian $\Delta_k$ equals
\begin{align*}
\Delta_k\,f(X)&=\sum_{i=1}^d\,T_{{\bf e}_i}^2\,f(X)\\
&=\Delta\,f(X)+2\,\sum_{\alpha\in\Sigma^+}\,k(\alpha)\,\left[\frac{\langle \alpha,\nabla f(X)\rangle}{\langle \alpha,X\rangle}-\frac{f(X)-f(\sigma_\alpha\,X)}{\langle \alpha,X\rangle^2}\right].
\end{align*}
Here $\Delta$ and $\nabla$ denote the regular Laplacian and gradient.
  
The formula \eqref{heatSpher} remains true for the $W$-invariant kernels  $p^W_t$  and   $E^W$ and translates  
in a similar relationship between the spherical function $\psi_\lambda$ and the heat kernel $p_t^W(X,Y)$:
\begin{align}\label{heatSpher1}
p_t^W(X,Y)&=\frac1{2^{\gamma+d/2} c_k} \,t^{-\frac{d}{2}-\gamma} \,e^{\frac{-|X|^2-|Y|^2}{4t}}\,\psi_X\left(\frac{Y}{2t}\right).
\end{align}

In Section \ref{comp}, we prove the Conjecture \ref{C}
for the root system $A_n$, with an arbitrary
multiplicity $k>0$, i.e.{} we prove the formula \eqref{CC}
 providing exact estimates for the spherical functions  $\psi_\lambda(X)$ in the two variables $X$, $\lambda$  when $\lambda$ is real.

In Section \ref{appHeat}, we apply the sharp estimates
\eqref{CC} of the spherical functions $\psi_\lambda(X)$
to the $W$-invariant Dunkl heat kernel $p_t(X,Y)$ for the root system $A_n$, with an arbitrary
multiplicity $k>0$. In the Theorem  \ref{heat}, we obtain
sharp estimates of $p_t(X,Y)$ in three variables $t,X,Y$.

Next, in Sections \ref{appNewton} and \ref{appFract}, we apply the Theorem  \ref{heat} to the $W$-invariant Dunkl Newton kernel and to the  $W$-invariant $s$-stable semigroups, respectively.  In all cases, we obtain sharp estimates.

\section{ Proof of the Conjecture in the case $A_n$.}\label{comp}

We will assume from now on that $X\in\a^+$ and $\lambda\in\overline{\a^+}$.

\begin{Note}
We will write $f(x)\lesssim g(x)$ ($f(x)\gtrsim g(x)$) for $x\in D$ if there exists a constant $C>0$ independent of $x$ such that $f(x)\leq C\,g(x)$ ($f(x)\geq C\,g(x)$) for all $x\in D$.
We will use the notation $M_k=(x_k+x_{k+1})/2$.
\end{Note}

\begin{remark}
Suppose $i<j$. We will use repeatedly the fact that the functions
\begin{align*}
\frac{x}{1+(\lambda_i-\lambda_j)\,x} \qquad \text{and} \qquad 
\frac{x}{(1+(\lambda_i-\lambda_j)\,x)^{k}}, \quad  k\leq 1
 \end{align*}
 are increasing functions of $x$. A feature of our proofs will be the distinction between the cases $0<k\leq 1$ and $k>1$.
\end{remark}

\begin{proposition}\label{induc}
Conjecture \ref{C} is equivalent to
\begin{align}\label{eq:I(n)}
I^{(n)}\asymp\frac{\pi(X)^{2\,k-1}}{\prod_{i<j\leq n+1}\,((1+(\lambda_i-\lambda_j)(x_i-x_j))^{k}}
\end{align}
where 
\begin{align*}
I^{(n)}&=
\int_{x_{n+1}}^{x_n}\,\dots\int_{x_2}^{x_1}
\,e^{-\sum_{i=1}^n\,(\lambda_i-\lambda_{n+1})\,(x_i-y_i)}
\\&\qquad
\,\left(\prod_{i\leq j\leq n}\,(x_i-y_j)\,\prod_{i< j\leq n+1}\,(y_i-x_j)\right)^{k-1}
\\&\qquad
\,\prod_{i<j\leq n}\,\frac{y_i-y_j}{(1+(\lambda_i-\lambda_j)(y_i-y_j))^{k}} 
\,dy_1\dots dy_n.
\end{align*}
\end{proposition}

\begin{proof}
The  integral $I^{(n)}$ corresponds to a constant multiple of $e^{-\lambda(X)}\,\pi(X)^{2\,k-1}\,\psi_\lambda(e^X)$ in which we have replaced $\psi_{\lambda_0}(e^Y)$ in \eqref{iter} by its asymptotic expression conjectured in \eqref{CC}.
\end{proof}

We start by two technical results.

\begin{lemma}\label{A}
For $k>0$ and $x\geq0$, we have
\begin{align*}
\int_{0}^x\,u^{k-1}\,e^{-u} \,du\asymp\left(\frac{x}{1+x}\right)^{k}.
\end{align*}
\end{lemma}

\begin{proof}
The result is clearly true if $0\leq x<1$ (use $e^{-1}\le  e^{-x}\le 1$  and integrate).  If $x\geq1$ then
\begin{align*}
\int_{0}^1\,u^{k-1}\,e^{-u} \,du\leq \int_{0}^x\,u^{k-1}\,e^{-u} \,du<\int_{0}^\infty\,u^{k-1}\,e^{-u} \,du
\end{align*}
and the result follows.
\end{proof}

\begin{proposition}\label{truncated}
Assume that $\gamma=x_n-x_{n+1}$ is the largest positive root and let
\begin{align*}
I_1&=
\int_{M_n}^{x_n}\,\dots\int_{x_2}^{x_1}
\,e^{-\sum_{i=1}^n\,(\lambda_i-\lambda_{n+1})\,(x_i-y_i)}
\,\left(\prod_{i\leq j\leq n}\,(x_i-y_j)\,\prod_{i< j\leq n+1}\,(y_i-x_j)\right)^{k-1}
\\&\qquad
\,\prod_{i<j\leq n}\,\frac{y_i-y_j}{(1+(\lambda_i-\lambda_j)(y_i-y_j))^{k}} 
\,dy_1\dots dy_n.
\end{align*}
Then $I_1\asymp I^{(n)}$.
\end{proposition}

\begin{proof}
Let $I_2=I^{(n)}-I_1$.
In $I_1$ and $I_2$, consider only the corresponding integral in $y_n$, calling the resulting expressions $\tilde{I}_1$ and  $\tilde{I}_2$. 
Observing that $y_n-x_{n+1}\asymp\gamma$ for $y_n\in[x_n,M_n]$, we have
\begin{align*}
\tilde{I}_1&\asymp\gamma^{k-1}\,\int_{M_n}^{x_n}
\,e^{-(\lambda_n-\lambda_{n+1})\,(x_n-y_n)}
\,\left(\prod_{i\leq n}\,(x_i-y_n)\right)^{k-1}
\\&\qquad\qquad\qquad
\,\prod_{i<n}\,\frac{y_i-y_n}{(1+(\lambda_i-\lambda_n)(y_i-y_n))^{k}}
dy_n.
\end{align*}

If $0<k\leq 1$ then, 
\begin{align*}
\tilde{I}_1&\gtrsim\,e^{-(\lambda_n-\lambda_{n+1})\,\gamma/2}\gamma^{k-1}\,\int_{M_n}^{x_n}
\,(x_1-y_n)^{k-1}\,\prod_{i=2}^n\,(x_i-y_n)^{k-1}
\,\prod_{i<n}\,\frac{x_{i+1}-y_n}{(1+(\lambda_i-\lambda_n)\,\gamma)^{k}}
dy_n\\
&\gtrsim\gamma^{k-1}\,e^{-(\lambda_n-\lambda_{n+1})\,\gamma/2}\,\int_{M_n}^{x_n}
\,\prod_{i<n}\,\frac{(x_{i+1}-y_n)^{k}}{(1+(\lambda_i-\lambda_n)\,\gamma)^{k}} \,(x_1-y_n)^{k-1}\,
dy_n\\
&\gtrsim\gamma^{k-1}
\,\frac{e^{-(\lambda_n-\lambda_{n+1})\,\gamma/2}}{\prod_{i<n}\,(1+(\lambda_i-\lambda_n)\,\gamma)^{k}}
\,\int_{M_n}^{x_n}\,(x_n-y_n)^{(n-1)\,k} \,(x_1-y_n)^{k-1}\,
dy_n\\
&\gtrsim\gamma^{(n+1)\,k-1}\,\frac{e^{-(\lambda_n-\lambda_{n+1})\,\gamma/2}}{\prod_{i<n}\,(1+(\lambda_i-\lambda_n)\,\gamma)^{k}}.
\end{align*}
Indeed, if $n>1$ then $(x_1-y_n)^{k-1}\gtrsim \gamma^{k-1}$ and the rest can easily be integrated. If $n=1$, then
\begin{align*}
\int_{M_n}^{x_n}\,(x_n-y_n)^{(n-1)\,k} \,(x_1-y_n)^{k-1}\,
dy_n=\int_{M_n}^{x_n}\,(x_n-y_n)^{\,k-1} \,dy_n\asymp \gamma^{k}.
\end{align*}

If $k>1$, we have
\begin{align*}
\tilde{I}_1&\gtrsim \gamma^{k-1}\,e^{-(\lambda_n-\lambda_{n+1})\,\gamma/2}\,\int_{M_n}^{x_n}
\,\prod_{i=1}^{n-1}\,(x_i-y_n)^{k-1}
\,\prod_{i<n}\,\frac{y_i-y_n}{(1+(\lambda_i-\lambda_n)(y_i-y_n)^{k}}
\\&\qquad\qquad\qquad\qquad\qquad
\,(x_n-y_n)^{k-1}
dy_n\\
&\gtrsim \gamma^{k-1}\,e^{-(\lambda_n-\lambda_{n+1})\,\gamma/2}\,\int_{M_n}^{x_n}
\,\left(\prod_{i<n}\,\frac{y_i-y_n}{1+(\lambda_i-\lambda_n)(y_i-y_n)} \right)^{k}
\,(x_n-y_n)^{k-1}\,dy_n\\
&\gtrsim  \gamma^{k-1}\,\frac{e^{-(\lambda_n-\lambda_{n+1})\,\gamma/2}}{\prod_{i<n}\,(1+(\lambda_i-\lambda_n)\,\gamma)^{k}}
\,\int_{M_n}^{x_n}\prod_{i<n}\,(y_i-y_n)^{k}\,\,(x_n-y_n)^{k-1}\,dy_n\\
&\gtrsim  \gamma^{k-1}\,\frac{e^{-(\lambda_n-\lambda_{n+1})\,\gamma/2}}{\prod_{i<n}\,(1+(\lambda_i-\lambda_n)\,\gamma)^{k}}
\,\int_{M_n}^{x_n}\,(x_n-y_n)^{n\,k-1}\,dy_n\\
&\asymp \gamma^{(n+1)\,k-1}\,\frac{e^{-(\lambda_n-\lambda_{n+1})\,\gamma/2}}{\prod_{i<n}\,(1+(\lambda_i-\lambda_n)\,\gamma)^{k}}.
\end{align*}

On the other hand for $k>0$, observing that $x_i-y_n\asymp\gamma$ and $y_i-y_n\asymp \gamma$ for $y_n\in[M_n,x_{n+1}]$, we have
\begin{align*}
\tilde{I}_2&\asymp\gamma^{n\,(k-1)}\,\prod_{i<n}\,\frac{\gamma}{(1+(\lambda_i-\lambda_n)\gamma)^{k}} \,\int_{x_{n+1}}^{M_n}
\,e^{-(\lambda_n-\lambda_{n+1})\,(x_n-y_n)}
\,(y_n-x_{n+1})^{k-1}\,dy_n\\\\
&\lesssim \gamma^{n\,(k-1)+n-1}\,\frac{e^{-(\lambda_n-\lambda_{n+1})\,\gamma/2}}{\prod_{i<n}\,(1+(\lambda_i-\lambda_n)\gamma)^{k}} \,\int_{x_{n+1}}^{M_n}\,(y_n-x_{n+1})^{k-1}\,dy_n\\
&\asymp\gamma^{n\,(k-1)+n-1+k}\,\frac{e^{-(\lambda_n-\lambda_{n+1})\,\gamma/2}}{\prod_{i<n}\,(1+(\lambda_i-\lambda_n)\gamma)^{k}}
\lesssim \tilde{I}_1
\end{align*}
which allows us to conclude.
\end{proof}

\begin{theorem}\label{main}
Conjecture \ref{C} holds for the root system $A_n$, $n\geq 1$ with root multiplicity $k>0$.
\end{theorem}

\begin{proof}
The result is proven using induction.  Using Proposition \ref{truncated} with $n=1$,
we have, using $u=(\lambda_1-\lambda_2)\,(x_1-y_1)$ and Lemma \ref{A},
\begin{align*}
I^{(1)}&\asymp I_1= \int_{M_1}^{x_1}\,e^{-(\lambda_1-\lambda_2)\,(x_1-y_1)}\,(x_1-y_1)^{k-1}\,(y_1-x_2)^{k-1}\,dy_1\\
&\asymp(x_1-x_2)^{k-1}\,\int_{M_1}^{x_1}\,e^{-(\lambda_1-\lambda_2)\,(x_1-y_1)}\,(x_1-y_1)^{k-1}\,dy_1\\
&=(x_1-x_2)^{k-1}\,(\lambda_1-\lambda_2)^{-k}\,\int_0^{(\lambda_1-\lambda_2)\,(x_1-x_2)/2}\,e^{-u}\,u^{k-1}\,du\\
&\asymp(x_1-x_2)^{k-1}\,\left(\frac{(x_1-x_2)/2}{1+(\lambda_1-\lambda_2)\,(x_1-x_2)/2}\right)^{k}
\end{align*}
which proves the formula \eqref{eq:I(n)} in the case $n=1$.
\\

Assume that the result holds for the root systems $A_1$, $A_2$, \ldots, $A_{n-1}$.  We will use Proposition \ref{induc} and will proceed by assuming, in turn for each $m<n$, that $\alpha_m=x_m-x_{m+1}$ is the largest root. We will discuss the case $m=n$ at the end. 

We will proceed as follows.  As in the proof of Proposition \ref{truncated}, we will divide the integral in two parts $I_1$ and $I_2$, show that $I_1$ has the desired asymptotics and that $I_2\lesssim I_1$.

Assume now that $\alpha_m=x_m-x_{m+1}$, $1\leq m\leq n-1$, is the largest root. 
Noting that $x_i-y_j\asymp \alpha_m$, $i\leq m$, $m<j\leq n$, $y_i-x_j\asymp \alpha_m$, $i\leq m$, $j\geq m+2$, $y_i-y_j\asymp \alpha_m$, $i\leq m$, $m<j\leq n$, for $y_m\in[M_m,x_m]$,  we have
\begin{align*}
I_1&=\int_{x_{n+1}}^{x_n}\dots\int_{M_m}^{x_m}\dots\int_{x_2}^{x_1}\,\,e^{-\sum_{i=1}^n\,(\lambda_i-\lambda_{n+1})\,(x_i-y_i)}
\\&\qquad
\,\left(\prod_{i\leq j\leq m}\,(x_i-y_j)\,\prod_{i< j\leq m+1}\,(y_i-x_j)\right)^{k-1}
\\&\qquad
\,\left(\prod_{m<i\leq j\leq n}\,(x_i-y_j)\,\prod_{m<i< j\leq n+1}\,(y_i-x_j)\right)^{k-1}
\\&\qquad
\,\left(\prod_{i\leq m<j\leq n}\,(x_i-y_j)\,\prod_{\genfrac{}{}{0pt}{}{i\leq m,}{m+1<j\leq n+1}}\,(y_i-x_j)\right)^{k-1}
\\&\qquad
\,\prod_{i<j\leq m}\,\frac{y_i-y_j}{(1+(\lambda_i-\lambda_j)(y_i-y_j))^{k}} 
\,\prod_{m<i<j\leq n}\,\frac{y_i-y_j}{(1+(\lambda_i-\lambda_j)(y_i-y_j))^{k}} 
\\&\qquad
\,\prod_{i\leq m<j\leq n}\,\frac{y_i-y_j}{(1+(\lambda_i-\lambda_j)(y_i-y_j))^{k}} 
\,dy_1\dots dy_n\\
&\asymp
\alpha_m^{2\,m\,(n-m)\,\,(k-1)}\,\prod_{i\leq m<j\leq n}\,\frac{\alpha_m}{(1+(\lambda_i-\lambda_j)\alpha_m)^{k}} 
\\&\qquad
\,
\int_{x_{n+1}}^{x_n}\dots\int_{M_m}^{x_m}\dots\int_{x_2}^{x_1}
\,e^{-\sum_{i=1}^n\,(\lambda_i-\lambda_{n+1})\,(x_i-y_i)}
\\&\qquad
\,\left(\prod_{i\leq j\leq m}\,(x_i-y_j)\,\prod_{i< j\leq m+1}\,(y_i-x_j)\right)^{k-1}
\\&\qquad   
\,\left(\prod_{m<i\leq j\leq n}\,(x_i-y_j)\,\prod_{m<i< j\leq n+1}\,(y_i-x_j)\right)^{k-1}
\\&\qquad
\,\prod_{i<j\leq m}\,\frac{y_i-y_j}{(1+(\lambda_i-\lambda_j)(y_i-y_j))^{k}} 
\,\prod_{m<i<j\leq n}\,\frac{y_i-y_j}{(1+(\lambda_i-\lambda_j)(y_i-y_j))^{k}} 
\,dy_1\dots dy_n\\
&=\frac{\alpha_m^{2\,m\,(n-m)\,(k-1)+m\,(n-m)}}
{\prod_{i\leq m<j\leq  n}\,(1+(\lambda_i-\lambda_j)\alpha_m)^{k}} \,
\int_{M_m}^{x_m}\dots\int_{x_2}^{x_1}\,e^{-\sum_{i=1}^m\,(\lambda_i-\lambda_{n+1})\,(x_i-y_i)}
\\&\qquad
\,\left(\prod_{i\leq j\leq m}\,(x_i-y_j)\,\prod_{i< j\leq m+1}\,(y_i-x_j)\right)^{k-1}
\\&\qquad
\,\prod_{i<j\leq m}\,\frac{y_i-y_j}{(1+(\lambda_i-\lambda_j)(y_i-y_j))^{k}} 
\,dy_1\dots dy_m
\\&\qquad
\,\int_{x_{n+1}}^{x_n}\dots\int_{x_{m+2}}^{x_{m+1}}
\,e^{-\sum_{i=m+1}^n\,(\lambda_i-\lambda_{n+1})\,(x_i-y_i)}
\\&\qquad
\,\left(\prod_{m<i\leq j\leq n}\,(x_i-y_j)\,\prod_{m<i< j\leq n+1}\,(y_i-x_j)\right)^{k-1}
\\&\qquad
\,\prod_{m<i<j\leq n}\,\frac{y_i-y_j}{(1+(\lambda_i-\lambda_j)(y_i-y_j))^{k}} 
\,dy_{m+1}\dots dy_n\\
&\asymp\frac{\alpha_m^{m\,(n-m)\,
\,(2\,k-1)}}{\prod_{i\leq m<j\leq n}\,(1+(\lambda_i-\lambda_j)\,\alpha_m)^{k}} 
\\&\qquad
\,\frac{\prod_{i<j\leq m}\,(x_i-x_j)^{2\,k-1}}{\prod_{i<j\leq m}\,(1+(\lambda_i-\lambda_j)\,(x_i-x_j))^{k}
\,\prod_{i\leq m}\,(1+(\lambda_i-\lambda_{n+1})\,(x_i-x_{m+1}))^{k}}
\\&\qquad
\,\prod_{m<i<j\leq n+1}\,\frac{(x_i-x_j)^{2\,k-1}}{(1+(\lambda_i-\lambda_j)\,(x_i-x_j))^{k}}
\end{align*}
which has the desired asymptotics (we used Proposition \ref{truncated} and the induction hypothesis on $A_m$ and on $A_{n-m}$).

It remains to show that $I_2=I^{(n)}-I_1\lesssim I_1$.  As in the proof of Proposition \ref{truncated},
it suffices to show that $\tilde{I}_1\gtrsim\tilde{I}_2$ where $\tilde{I}_1$ (respectively $\tilde{I}_2$) is the portion of $I_1$ ($I_2$) integrated with respect to $y_m$.

Now, since $y_m-x_j\asymp \alpha_m$, $m<j\leq n+1$,  and $y_m-y_j\asymp \alpha_m$, $m<j\leq n$,  when $y_m\in[M_m,x_m]$, we have
\begin{align*}
\tilde{I}_1
&\asymp \frac{\alpha_m^{(n+1-m)\,(k-1)+n-m}}{\prod_{m<j\leq n}\,(1+(\lambda_m-\lambda_j)\,\alpha_m)^{k}}
\,\int_{M_m}^{x_m}\,e^{-(\lambda_m-\lambda_{n+1})\,(x_m-y_m)}\,\prod_{i\leq m}\,(x_i-y_m)^{k-1}
\\&\qquad
\,\prod_{i<m}\,\frac{y_i-y_m}{(1+(\lambda_i-\lambda_m)\,(y_i-y_m))^{k}}\,dy_m.
\end{align*}

If $k>1$ then 
\begin{align*}
\tilde{I}_1
&\gtrsim\frac{\alpha_m^{(n+1-m)\,(k-1)+n-m}}{\prod_{m<j\leq n}\,(1+(\lambda_m-\lambda_j)\,\alpha_m)^{k}}
\,\,e^{-(\lambda_m-\lambda_{n+1})\,\alpha_m/2}\,\int_{M_m}^{x_m}\,\prod_{i\leq m}\,(x_m-y_m)^{k-1}
\\&\qquad
\,\prod_{i<m}\,\frac{x_m-y_m}{(1+(\lambda_i-\lambda_m)\,\alpha_m)^{k}}\,dy_m\\
&\gtrsim\frac{\alpha_m^{(n+1-m)\,(k-1)+n-m}}{\prod_{m<j\leq n}\,(1+(\lambda_m-\lambda_j)\,\alpha_m)^{k}}
\,\frac{e^{-(\lambda_m-\lambda_{n+1})\,\alpha_m/2}}{\prod_{i<m}\,(1+(\lambda_i-\lambda_m)\,\alpha_m)^{k}}
\\&\qquad
\,\int_{M_m}^{x_m}\,(x_m-y_m)^{m\,(k-1)+m-1}\,dy_m\\
&=\frac{\alpha_m^{(n+1)\,k-1}}{\prod_{m<j\leq n}\,(1+(\lambda_m-\lambda_j)\,\alpha_m)^{k}}
\,\frac{e^{-(\lambda_m-\lambda_{n+1})\,\alpha_m/2}}{\prod_{i<m}\,(1+(\lambda_i-\lambda_m)\,\alpha_m)^{k}}.
\end{align*}

If $0<k\leq 1$ then
\begin{align*}
\tilde{I}_1
&\gtrsim\frac{\alpha_m^{(n+1-m)\,(k-1)+n-m}}{\prod_{m<j\leq n}\,(1+(\lambda_m-\lambda_j)\,\alpha_m)^{k}}
\,e^{-(\lambda_m-\lambda_{n+1})\,\alpha_m/2}
\,\int_{M_m}^{x_m}\,\prod_{i< m}\,\alpha_m^{k-1}\,(x_m-y_m)^{k-1}
\\&\qquad
\,\prod_{i<m}\,\frac{x_m-y_m}{(1+(\lambda_i-\lambda_m)\,\alpha_m)^{k}}\,dy_m\\
&=\frac{\alpha_m^{(n+1-m)\,(k-1)+n-m+(m-1)\,(k-1)}}{\prod_{m<j\leq n}\,(1+(\lambda_m-\lambda_j)\,\alpha_m)^{k}}
\,\frac{e^{-(\lambda_m-\lambda_{n+1})\,\alpha_m/2}}{\prod_{i<m}\,(1+(\lambda_i-\lambda_m)\,\alpha_m)^{k}}
\\&\qquad
\,\int_{M_m}^{x_m}\,(x_m-y_m)^{m-1+k-1}\,dy_m\\
&\asymp \frac{\alpha_m^{(n+1-m)\,(k-1)+n-m+(m-1)\,(k-1)}}{\prod_{m<j\leq n}\,(1+(\lambda_m-\lambda_j)\,\alpha_m)^{k}}
\,\frac{e^{-(\lambda_m-\lambda_{n+1})\,\alpha_m/2}}{\prod_{i<m}\,(1+(\lambda_i-\lambda_m)\,\alpha_m)^{k}}
\,\alpha_m^{m-1+k}\\
&= \frac{\alpha_m^{(n+1)\,k-1}}{\prod_{m<j\leq n}\,(1+(\lambda_m-\lambda_j)\,\alpha_m)^{k}}
\,\frac{e^{-(\lambda_m-\lambda_{n+1})\,\alpha_m/2}}{\prod_{i<m}\,(1+(\lambda_i-\lambda_m)\,\alpha_m)^{k}}.
\end{align*}

On the other hand, since $x_i-y_m\asymp \alpha_m$, $i\leq m$, and $y_i-y_m\asymp \alpha_m$, $i<m$, when $y_m\in[x_{m+1},M_m]$,
\begin{align*}
\tilde{I}_2
&\asymp \alpha_m^{m\,(k-1)}\,\prod_{i<m}\,\frac{\alpha_m}{(1+(\lambda_i-\lambda_m)\,\alpha_m)^{k}}
\,\int_{x_{m+1}}^{M_m}\,e^{-(\lambda_m-\lambda_{n+1})\,(x_m-y_m)}
\\&\qquad
\,\prod_{m<j\leq n+1}\,(y_m-x_j)^{k-1}
\,\prod_{m<j\leq n}\,\frac{y_m-y_j}{(1+(\lambda_m-\lambda_j)\,(y_m-y_j))^{k}}\,dy_m.
\end{align*}

If $k>1$ then
\begin{align*}
\tilde{I}_2
&\lesssim \alpha_m^{m\,(k-1)}\,\prod_{i<m}\,\frac{\alpha_m}{(1+(\lambda_i-\lambda_m)\,\alpha_m)^{k}}
\,e^{-(\lambda_m-\lambda_{n+1})\,\alpha_m/2}
\,\int_{x_{m+1}}^{M_m}\,\prod_{m<j\leq n}\,(y_m-y_j)^{k-1}
\\&\qquad
\,\prod_{m<j\leq n}\,\frac{y_m-y_j}{(1+(\lambda_m-\lambda_j)\,(y_m-y_j))^{k}}\,(y_m-x_{n+1})^{k-1}\,dy_m\\
&=\alpha_m^{m\,(k-1)}\,\prod_{i<m}\,\frac{\alpha_m}{(1+(\lambda_i-\lambda_m)\,\alpha_m)^{k}}
\,e^{-(\lambda_m-\lambda_{n+1})\,\alpha_m/2}
\\&\qquad
\,\int_{x_{m+1}}^{M_m}
\,\prod_{m<j\leq n}\,\left(\frac{y_m-y_j}{1+(\lambda_m-\lambda_j)\,(y_m-y_j)}\right)^{k}\,(y_m-x_{n+1})^{k-1}\,dy_m\\
&\lesssim \alpha_m^{m\,(k-1)}\,\prod_{i<m}\,\frac{\alpha_m}{(1+(\lambda_i-\lambda_m)\,\alpha_m)^{k}}
\,e^{-(\lambda_m-\lambda_{n+1})\,\alpha_m/2}\,\alpha_m^{k}
\\&\qquad
\,\prod_{m<j\leq n}\,\left(\frac{\alpha_m}{1+(\lambda_m-\lambda_j)\,\alpha_m}\right)^{k}
\lesssim \tilde{I}_1.
\end{align*}

If $0<k\leq 1$ then
\begin{align*}
\tilde{I}_2
&\lesssim\alpha_m^{m\,(k-1)}\,\prod_{i<m}\,\frac{\alpha_m}{(1+(\lambda_i-\lambda_m)\,\alpha_m)^{k}}
\,e^{-(\lambda_m-\lambda_{n+1})\,\alpha_m/2}
\,\int_{x_{m+1}}^{M_m}\,\prod_{m+1<j\leq n+1}\,(y_m-x_j)^{k-1}
\\&\qquad
\,\prod_{m<j\leq n}\,\frac{y_m-x_{j+1}}{(1+(\lambda_m-\lambda_j)\,(y_m-x_{j+1}))^{k}}
\,(y_m-x_{m+1})^{k-1}\,dy_m\\
&=\alpha_m^{m\,(k-1)}\,\prod_{i<m}\,\frac{\alpha_m}{(1+(\lambda_i-\lambda_m)\,\alpha_m)^{k}}
\\&\qquad
\,e^{-(\lambda_m-\lambda_{n+1})\,\alpha_m/2}\,\int_{x_{m+1}}^{M_m}
\,\left(\prod_{m<j\leq n}\,\frac{y_m-x_{j+1}}{1+(\lambda_m-\lambda_j)\,(y_m-x_{j+1})}\right)^{k}
\\&\qquad
\,(y_m-x_{m+1})^{k-1}\,dy_m\\
&\lesssim \alpha_m^{m\,(k-1)}\,\prod_{i<m}\,\frac{\alpha_m}{(1+(\lambda_i-\lambda_m)\,\alpha_m)^{k}}
\,e^{-(\lambda_m-\lambda_{n+1})\,\alpha_m/2}\,\alpha_m^{k}
\\&\qquad
\,\left(\prod_{m<j\leq n}\,\frac{\alpha_m}{1+(\lambda_m-\lambda_j)\,\alpha_m}\right)^{k}
\lesssim \tilde{I}_1.
\end{align*}

By the structure of the root system $A_n$, the case $\alpha_n$ maximal is equivalent to the case $\alpha_1$ maximal.  Indeed, in formula \eqref{iter}, one does not assume that $\lambda\in\overline{\a^+}$.  We also know that $\psi_\lambda(e^X)$ is invariant under permutation of its $\lambda$ argument.  Hence one can re-write \eqref{iter}
by exchanging $\lambda_1$ and $\lambda_{n+1}$,
\begin{align*}
\psi_\lambda(e^X)&=e^{\lambda(X)}\ \hbox{if $n=1$ and}\nonumber\\
\psi_\lambda(e^X)&=\frac{\Gamma(k\,(n+1))}{(\Gamma(k))^{n+1}}\,e^{\lambda_1\,\sum_{r=1}^{n+1}\,x_r}\,(\prod_{i<j\leq n+1}\,(x_i-x_j))^{1-2\,k}\,\int_{x_{n+1}}^{x_n}\,\cdots\,\int_{x_2}^{x_1}
\,\psi_{\widetilde{\lambda_0}}(e^Y)
\\&\left[\prod_{i=1}^n\,\left(\prod_{j=1}^i\,(x_j-y_i)\,\prod_{j=i+1}^{n+1}\,(y_i-x_j)\right)\right]^{k-1}
\,\prod_{i<j\leq n}\,(y_i-y_j)\,dy_1\cdots dy_n
\end{align*}
where $\widetilde{\lambda_0}(U)=\sum_{r=2}^{n+1}\,(\lambda_r-\lambda_1)\,u_r$. We used the fact that
\begin{align*}
\psi_{[\lambda_{n+1}-\lambda_1,\lambda_2-\lambda_1,\dots,\lambda_n-\lambda_1]}(e^Y)
=\psi_{[\lambda_2-\lambda_1,\dots,\lambda_n-\lambda_1,\lambda_{n+1}-\lambda_1]}(e^Y).
\end{align*}

Conjecture \ref{C} is equivalent to
\begin{align*}
J^{(n)}\asymp
\frac{\pi(X)^{2\,k-1}}{\prod_{i<j\leq n+1}\,((1+(\lambda_i-\lambda_j)(x_i-x_j))^{k}}
\end{align*}
where
\begin{align*}
J^{(n)}&=
\int_{x_{n+1}}^{x_n}\,\dots\int_{x_2}^{x_1}
\,e^{-\sum_{i=1}^{n}\,(
\lambda_1-\lambda_{i+1})\,(y_i-x_{i+1})}
\,\left(\prod_{i\leq j\leq n}\,(x_i-y_j)\,\prod_{i< j\leq n+1}\,(y_i-x_j)\right)^{k-1}
\\&\qquad
\,\prod_{i<j\leq n}\,\frac{y_i-y_j}{(1+(\lambda_{i+1}-\lambda_{j+1})(y_i-y_j))^{k}} 
\,dy_1\dots dy_n.
\end{align*}
The term $J^{(n)}$ corresponds to a constant multiple of $e^{-\lambda(X)}\,\pi(X)^{2\,k-1}\,\psi_\lambda(e^X)$ in which we have replaced $\psi_{\lambda_0}(e^Y)$ in \eqref{iter} by its asymptotic expression conjectured  in \eqref{CC}.  One then proves the case $\alpha_n$ maximal as one proves the case $\alpha_1$ maximal.

This concludes the proof of the estimate \eqref{CC} for $X\in\a^+$ (recall that the formula \eqref{iter} holds for $X\in\a^+$). The estimates that we find for $\psi_\lambda(e^X)$ extend to $X\in\overline{\a^+}$ by continuity.
\end{proof}

\section{Applications}\label{app}
\subsection{Estimates of the $W$-invariant Dunkl Heat Kernel}\label{appHeat}

The following theorem  establishes,  for root systems $A_n$  and  for any multiplicity $k>0$,  the estimates of the $W$-invariant Dunkl Heat Kernel conjectured in the
Conjecture 18 of \cite{PGPS1}.

\begin{theorem}\label{heat}
For the root systems of type $A$, we have for $X$, $Y\in \overline{\a^+}$
\begin{align*}
p^W_t(X,Y)\asymp  \frac{t^{-d/2}\,e^{-|X-Y|^2/(4\,t)}}{\prod_{\alpha>0}\,(t+\alpha(X)\,\alpha(Y))^{k}}.
\end{align*}
\end{theorem}

\begin{proof}
Consider the relation \eqref{heatSpher1}
\begin{align*}
p_t^W(X,Y)=t^{-d/2-\gamma}\,e^{-|X-Y|^2/(4\,t)}\,\psi_X(Y/(2\,t))
\end{align*}
with $\gamma=\,\sum_{\alpha>0}\,k(\alpha)= k|\Sigma^+|$.  From Theorem \ref{main}, we have
\begin{align*}
p_t^W(X,Y)&\asymp t^{-d/2-\gamma}\,e^{-(|X|^2+|Y|^2)/(4\,t)}\,\frac{e^{\langle X,Y/(2\,t)\rangle}}{\prod_{\alpha>0}\,(1+\alpha(X)\,\alpha(Y/(2\,t)))^{k}}\\
&=t^{-d/2-\gamma+k\,|\Sigma^+|}\,e^{-(|X|^2+|Y|^2)/(4\,t)}\,\frac{e^{\langle X,Y/(2\,t)\rangle}}{\prod_{\alpha>0}\,(2\,t+\alpha(X)\,\alpha(Y))^{k}}\\
&\asymp \frac{t^{-d/2}\,e^{-|X-Y|^2/(4\,t)}}{\prod_{\alpha>0}\,(t+\alpha(X)\,\alpha(Y))^{k}}.
\end{align*}
\end{proof}

\subsection{Estimates of the $W$-invariant Dunkl Newton Kernel}\label{appNewton}

The $W$-invariant Dunkl Newton kernel $N^W(X,Y)$ is the kernel of the inverse operator of the Dunkl Laplacian $\Delta^W$. It 
is the fundamental solution of the problem $\Delta^W\,u=f$ where $f$ is given and $|u(x)|\to0$ as $x\to\infty$. It is defined by 
\begin{align*}
N^W(X,Y)=\int_0^\infty\,p^W_t(X,Y)\,dt,
\end{align*}
where $p^W_t(X,Y)$ is the heat kernel of $\Delta^W$. 

In \cite{PGTLPS}, we stated the following conjecture for the Weyl invariant Newton kernel for $d\geq 3$ and proved it for complex root systems.
\begin{conjecture}\label{TLN} 
For $X$, $Y\in \overline{\a^+}$ and $d\geq 3$, we have
\begin{align*}
N^W(X,Y)\asymp \frac{1}{|X-Y|^{d-2}\,\prod_{\alpha\in \Sigma^+}|X-\sigma_\alpha Y|^{2\,k(\alpha)}}.
\end{align*}
\end{conjecture}

In this section, we prove the conjecture in the case of root systems of type $A$ and prove a similar result in the case $d=2$.\\

The next three lemmas will be useful to derive sharp estimates for the Newton kernel.

\begin{lemma}\label{ai}
Suppose $k>0$, $a\geq 0$, $b_i\geq 0$,
$a+b_i>0$, $i=1$, \dots, $m$ and $N>k\,m-1$.  Then
\begin{align*}
J:=\int_0^\infty\,\frac{u^N\,e^{-u}\,du}{\prod_{i=1}^m\,(a+b_i\,u)^k}
\asymp \frac{1}{\prod_{i=1}^m\,(a+b_i)^k}.
\end{align*}
\end{lemma}

\begin{proof}
Note that $a+b_i\,u\le (a+b_i)u $ whenever $u\geq 1$. 
Therefore, we have
\begin{align*}
J&\geq \int_1^\infty\,\frac{u^N\,e^{-u}\,du}{\prod_{i=1}^m\,(a+b_i\,u)^k}
\geq \frac{1}{\prod_{i=1}^m\,(a+b_i)^k}\,\int_1^\infty\,u^{N-m\,k}\,e^{-u}\,du.
\end{align*}
Let $\Lambda$ be the (possibly empty) set of indices where $a\leq b_i$.
\begin{align*}
J&\leq \frac{1}{\left(\prod_{i\not\in \Lambda}\,a\right)^k}\,\int_0^\infty\,\frac{u^N\,e^{-u}\,du}{\prod_{i\in\Lambda}\,(b_i\,u)^k}
=\frac{1}{\left(\prod_{i\not\in\Lambda}\,a\right)^k\,\left(\prod_{i\in\Lambda}\,b_i\right)^k}\,\int_0^\infty\,u^{N-|\Lambda|\,k}\,e^{-u}\,du\\
&\lesssim\frac{1}{\prod_{i=1}^k\,(a+b_i)^k}\,\max_{0\leq m_0\leq m}\,\int_0^\infty\,u^{N-m_0\,k}\,e^{-u}\,du
\end{align*}
(we understand an empty product to be equal to 1).
\end{proof}

\begin{lemma}\label{a1}
Suppose $k>0$, $a>0$ and $b\geq 0$. Then
\begin{align*}
J:=\int_0^\infty\,\frac{u^{k-1}\,e^{-u}\,du}{(a+b\,u)^k}\asymp \frac{\ln (2+b/a)}{(a+b)^k}.
\end{align*}
\end{lemma}

\begin{proof}
If $0\leq b\leq a$ then 
\begin{align*}
\int_0^\infty\,\frac{u^{k-1}\,e^{-u}\,du}{(a+a\,u)^k}\leq J\leq
\int_0^\infty\,\frac{u^{k-1}\,e^{-u}\,du}{a^k}
\end{align*}
and the result holds.  We now assume $a\leq b$.  We then have
\begin{align*}
J&\asymp\int_0^{a/b}\,\frac{u^{k-1}}{a^k}\,du+\int_{a/b}^1\,\frac{u^{k-1}\,du}{(b\,u)^k}+\int_1^\infty\,\frac{u^{k-1}\,e^{-u}\,du}{(b\,u)^k}\\
&\asymp\frac{1}{b^k}+ \frac{1}{b^k}\,\int_{a/b}^1\,u^{-1}\,du+  \frac{1}{b^k}\,\int_1^\infty\,u^{-1}\,e^{-u}\,du\\
&\asymp \frac{1}{b^k}+ \frac{1}{b^k}\,\ln (b/a)+ \frac{1}{b^k}\asymp \frac{\ln (2+b/a)}{b^k}
\end{align*}
which concludes the proof.
\end{proof}

\begin{lemma}\label{a2}
Suppose $k>0$, $a\geq 0$ and $0\leq b_1\leq b_2\leq b_3$ then
\begin{align*}
J:=\int_0^\infty\,\frac{u^{3\,k-1}\,e^{-u}\,du}{(a+b_1\,u)^k\,(a+b_2\,u)^k\,(a+b_3\,u)^k}
\asymp \frac{\ln(2+b_1/a)}{(a+b_1)^k\,(a+b_2)^k\,(a+b_3)^k}.
\end{align*}
\end{lemma}

\begin{proof}
If $b_1\leq a$ then
\begin{align*}
\int_1^\infty\,\frac{u^{3\,k-1}\,e^{-u}\,du}{(a+a\,u)^k\,((a+b_2)\,u)^k\,((a+b_3)\,u)^k}
\leq J
\leq \int_0^\infty\,\frac{u^{3\,k-1}\,e^{-u}\,du}{a^k\,(a+b_2\,u)^k\,(a+b_3\,u)^k}
\end{align*}
and the result follows in this case using Lemma \ref{ai} for the upper bound.

If $a\leq b_1\leq b_2\leq b_3$ then
\begin{align*}
J\gtrsim \int_{a/b_1}^2\,\frac{u^{3\,k-1}\,du}{(b_1\,u)^k\,(b_2\,u)^k\,(b_3\,u)^k}
\asymp \frac{\ln (2\,b_1/a)}{(b_1\,b_2\,b_3)^k}
\end{align*}
while
\begin{align*}
J&\lesssim \int_0^{a/b_1}\,\frac{u^{3\,k-1}\,du}{a^k\,(b_2\,u)^k\,(b_3\,u)^k}
+\int_{a/b_1}^2\,\frac{u^{3\,k-1}\,du}{(b_1\,u)^k\,(b_2\,u)^k\,(b_3\,u)^k}
+\int_2^\infty\,\frac{u^{3\,k-1}\,e^{-u}\,du}{(b_1\,u)^k\,(b_2\,u)^k\,(b_3\,u)^k}\\
&\asymp \frac{1}{b_1^k\,b_2^k\,b_3^k}+\frac{\ln(2\,b_1/a)}{b_1^k\,b_2^k\,b_3^k}+\frac{1}{b_1^k\,b_2^k\,b_3^k}
\end{align*}
and the result follows in this case.
\end{proof}

\begin{theorem}\label{Newton}
For the root system $A_n$ and $d\geq 3$, we have  for $X$, $Y\in \overline{\a^+}$
\begin{align*}
N^W(X,Y)\asymp \frac{|X-Y|^{2-d}}
{\prod_{\alpha>0}\,|X-\sigma_\alpha\,Y|^{2\,k}}.
\end{align*}
\end{theorem}

\begin{proof}
We have, using Theorem \ref{heat} and the change of variables $u=|X-Y|^2/(4\,t)$
\begin{align*}
N^W(X,Y)&=\int_0^\infty\,p_t^W(X,Y)\,dt
\asymp \int_0^\infty\,\frac{t^{-d/2}\,e^{-|X-Y|^2/(4\,t)}\,dt}{\prod_{\alpha>0}\,(t+\alpha(X)\,\alpha(Y))^{k}}\\
&\asymp |X-Y|^{2-d}
\,\int_0^\infty\,\frac{u^{d/2-2}\,e^{-u}\,du}
{\prod_{\alpha>0}\,(|X-Y|^2/(4\,u)+\alpha(X)\,\alpha(Y))^{k}}\\
&\asymp |X-Y|^{2-d}
\,\int_0^\infty\,\frac{u^{ k|\Sigma^+|+ d/2-2}\,e^{-u}\,du}
{\prod_{\alpha>0}\,(|X-Y|^2+\alpha(X)\,\alpha(Y)\,u)^{k}}
\\
&\asymp |X-Y|^{2-d}
\,\frac{1}
{\prod_{\alpha>0}\,(|X-Y|^2+\alpha(X)\,\alpha(Y))^{k}}
\asymp \frac{|X-Y|^{2-d}}
{\prod_{\alpha>0}\,|X-\sigma_\alpha\,Y|^{2\,k}}
\end{align*}
(we have used Lemma \ref{ai} and the fact that $|X-\sigma_\alpha\,Y|^2=|X-Y|^2+2\,\alpha(X)\,\alpha(Y)$).
\end{proof}

\begin{proposition}
If $d=2$, the Newton kernel in the $A_1$ case satisfies
\begin{align*}
N^W(X,Y)\asymp \frac{\ln\left(1+\frac{|X-\sigma_\alpha Y|^2}{|X-Y|^2}\right)}{|X-\sigma_\alpha Y|^{2\,k}}\qquad
X,Y\in \overline{\a^+}.
\end{align*}
\end{proposition}
Here, it is important to recall that for $X$, $Y\in\overline{\a^+}$, we have  $|X-\sigma_\alpha Y|\geq|X-Y|$ and therefore, the numerator of the last expression is at least $\ln 2$ for $X\not=Y$.
This remark also applies  to the estimate in Proposition \ref{d2A2}.

\begin{proof}
With computations similar as in the case $d\geq3$, using Lemma \ref{a1},
\begin{align*}
N^W(X,Y)&=\int_0^\infty\,\frac{t^{-1}\,e^{-|X-Y|^2/(4\,t)}\,dt}{(t+\alpha(X)\,\alpha(Y))^{k}}
\asymp \int_0^\infty\,\frac{u^{k-1}\,e^{-u}\,du}{(|X-Y|^2+\alpha(X)\,\alpha(Y)\,u)^{k}}\\
&\asymp\frac{\ln(2+\alpha(X)\,\alpha(Y)/|X-Y|^2)}{(|X-Y|^2+\alpha(X)\,\alpha(Y))^{k}}
=\frac{\ln\left(\frac{2\,|X-Y|^2+\alpha(X)\,\alpha(Y)}{|X-Y|^2}\right)}{(|X-Y|^2+\alpha(X)\,\alpha(Y))^{k}}\\
&=  \frac{\ln\left(\frac{3}{2}+ \frac12\frac{|X-\sigma_\alpha Y|^2}{|X-Y|^2}\right)}
{(|X-Y|^2+\alpha(X)\,\alpha(Y))^{k}}
\asymp 
\frac{\ln\left(1+\frac{|X-\sigma_\alpha Y|^2}{|X-Y|^2}\right)}{|X-\sigma_\alpha Y|^{2\,k}}.
\end{align*}
\end{proof}

\begin{proposition}\label{d2A2}
If $d=2$, the Newton kernel in the $A_2$ case satisfies
\begin{align*}
N^W(X,Y)\asymp \frac{\ln\left(1+\frac{|X-\sigma_\omega\,Y|^{2\,k}}{|X-Y|^2}\right)}{|X-\sigma_\alpha Y|^{2\,k}\,|X-\sigma_\beta Y|^2\,|X-\sigma_{\alpha+\beta} Y|^{2\,k}},\qquad X,Y\in \overline{\a^+},
\end{align*}
where $\omega$ gives the minimum of $|X-\sigma_\omega Y|$ for $\omega\in\{\alpha,\beta\}$.
\end{proposition}

\begin{proof}
With computations similar as in the case $d\geq3$, using Lemma \ref{a2},
\begin{align*}
N^W(X,Y)&\asymp\int_0^\infty\,\frac{t^{-1}\,e^{-|X-Y|^2/(4\,t)}\,dt}{(t+\alpha(X)\,\alpha(Y))^{k}\,(t+\beta(X)\,\beta(Y))^{k}\,(t+(\alpha+\beta)(X)\,(\alpha+\beta)(Y))^{k}}\\
&\asymp \int_0^\infty\,\frac{u^{3\,k-1}\,e^{-u}\,du}{\prod_{\eta\in\{\alpha,\beta,\alpha+\beta\}}(|X-Y|^2+\eta(X)\,\eta(Y)\,u)^{k}}\\
&\asymp \frac{\ln(2+\omega(X)\,\omega(Y)/|X-Y|^2)}{\prod_{\eta\in\{\alpha,\beta,\alpha+\beta\}}(|X-Y|^2+\eta(X)\,\eta(Y))^{k}}\\
&\asymp\frac{\ln\left(1+\frac{|X-\sigma_\omega\,Y|^2}{|X-Y|^2}\right)}{|X-\sigma_\alpha Y|^{2\,k}\,|X-\sigma_\beta Y|^{2\,k}\,|X-\sigma_{\alpha+\beta} Y|^{2\,k}}.
\end{align*}
where $\omega$ gives the minimum of $|X-\sigma_\omega Y|$ for $\omega\in\{\alpha,\beta\}$.
\end{proof}

\begin{remark}
In  the Dunkl  analysis, an  important role is played
by the intertwining operator $V_k$, defined
as a unique linear isomorphism  on the space of polynomial functions on $\R^d$  which intertwines the Dunkl operators with the usual partial derivatives:
\begin{align*}
T_\xi V_k = V_k \partial_\xi \quad \text{ for all } \xi \in \R^d
\end{align*}
and is normalized by $V_k(1)=1$.

The following general formula for the Dunkl Newton kernel  $N_k(x,y)$ involving the intertwining operator $V_k$ was proven in \cite{Gallardo2}:
\begin{align*}
N_k(X,Y)
&=\frac{2^{2\,\gamma}\,((d-2)/2)_\gamma}{|W|\,(d-2)\,w_d\,\pi(\rho)}\,V_k\left({(|Y|^2-2\,\langle X,\cdot\rangle+| X|^2)^{(2-d-2\,\gamma)/2}} \right)(Y)
\end{align*}
(we are using a slightly different normalization
of the operator $V_k$ than 
\cite{PGMRMY, Roesler}, see \cite{PGTLPS} for details.)
Little is known explicitly on the intertwining operator. 
Theorem \ref{Newton}  and the formula
 $
 N^W(X,,Y)=\frac{1}{|W|}\,\sum_{w\in W}\,N_k(w\,X,Y)$ imply the following explicit asymptotic formula.  
\begin{corollary}
For the root system $A_n$ and $d\geq 3$, we have  for $X$, $Y\in \overline{\a^+}$
\begin{align*}
V_k\left({(|Y|^2-\frac{2}{|W|}\,\sum_{w\in W}\,
\langle w\,X,\cdot\rangle+| X|^2)^{(2-d-2\,\gamma)/2}} \right)(Y)
\asymp \frac{|X-Y|^{2-d}}
{\prod_{\alpha>0}\,|X-\sigma_\alpha\,Y|^{2\,k}}.
\end{align*}
\end{corollary}
\end{remark}
%%%%%%%%%%%%%%%%%%%%%%%%%%%%%%%%%%%%%%%%%%%%%%
%%%%%%%%%%%%%%%%%%%%%%%%%%%%%%%%%%%%%%%%%%%%%%
\subsection{Heat semigroups for fractional powers of $\Delta_k^W$}
\label{appFract}

Let $s\in(0,2)$. The fractional powers $(-\Delta_k^W)^{s/2}$
of the $W$-invariant  Dunkl Laplacian are the infinitesimal generators of
 important semigroups $(h^W_t(X,Y))_{t\ge 0}$,
called  {\it $W$-invariant Dunkl $s$-stable semigroups}. 

Fractional powers of the   Dunkl Laplacian and related semigroups and processes were considered for $s=1$ in \cite[p.75]{RoeslerHAB}, 
\cite[Section 5]{RoeslerVoit} and for any $s\in(0,2)$ in \cite{Bou,Rejeb}.
Stable semigroups on Riemannian symmetric spaces of non-compact type were studied in \cite{Getoor, PG}.

Like the heat semigroup $p^W_t(X,Y)$, the densities $h^W_t(X,Y)$ are to be considered with respect to the Dunkl weight function $\omega_k(Y)$ on $\R^d$.  We have
\begin{align*}
h^W_t(X,Y)=\int_0^\infty\,p^W_u(X,Y)\,\eta_t(u)\,du
\end{align*}
where $\eta_t(u)$ is the density of the $s/2$-stable subordinator, i.e. of a positive L\'evy process $(Y_t)_{t>0}$   with the Laplace transform
${\bf E}\left(\exp(z\,Y_t)\right)=\exp(-t\,z^{s/2}),\, z>0$ 
(see \cite{Bertoin} for more details).

Denote by $h^{\R^d}_t(X,Y))_{t\ge 0}$ the $s$-stable rotationally  invariant semigroup on $\R^d$, with generator $(-\Delta)^{s/2}$.
It is known (\cite{BlGetoor}) that
\begin{align}
h^{\R^d}_t(X,Y)\asymp 
\min\left\{\frac{1}{t^{d/s}},\frac{t}{|X-Y|^{d+s}}\right\}
\asymp \frac{t}{(t^{2/s}+|X-Y|^2)^{(d+s)/2}}.\label{euclid}
\end{align}

\begin{remark}\label{min}
It is useful to note that 
$\min\left\{t^{-d/s},{t}{|X-Y|^{-(d+s)}}\right\}=t^{-d/s}$
if and only if $t^{2/s}\geq |X-Y|^2$. 
\end{remark}

\begin{theorem} \label{th:stable}
Consider the $W$-invariant  Dunkl Laplacian in the $A_n$ case with multiplicity $k>0$. Then for $X$, $Y\in \overline{\a^+}$,
\begin{align*}
h^W_t(X,Y)
\asymp
\frac{h^{\R^d}_t(X,Y)}{
\prod_{\alpha>0}  (t^{2/s} +|X-\sigma_\alpha Y|^2)^{k}}
\asymp
\frac{h^{\R^d}_t(X,Y)}{
\prod_{\alpha>0}  (t^{2/s} +|X- Y|^2 +\alpha(X)\alpha(Y))^{k}}.
\end{align*}
\end{theorem}

\begin{proof}
The proof is inspired by  the proof of \cite[Theorem 3.1]{Bogdan} providing estimates of stable semigroups on fractals.

Given Remark \ref{min}, it will make sense to consider the cases  $t^{2/s}\geq |X-Y|^2$ and  $t^{2/s}\leq |X-Y|^2$ separately.  In the proof, $m$ will denote the number of positive roots.

We start by showing that our estimate is an upper bound with an appropriate constant.  
In \cite[(14), page 168]{Bogdan}, it it shown that the subordinator density $\eta_t(u)$ satisfies
\begin{align}
\eta_t(u)\leq C\,t\,u^{-1-s/2}\,e^{-t\,u^{-s/2}}.\label{14}
\end{align}

Hence, using our estimates of the $W$-invariant Dunkl heat kernel in Theorem \ref{heat} and the change of variable $u=|X-Y|^2/(4\,w)$, we have
\begin{align*}
h^W_t(X,Y)
&\lesssim t\,\int_0^\infty\,\frac{u^{-d/2}\,e^{-|X-Y|^2/(4\,u)}\,u^{-1-s/2}\,e^{-t\,u^{-s/2}}\,du}{\prod_{\alpha>0}\,(u+\alpha(X)\,\alpha(Y))^{k}}
\\
&\lesssim t\,|X-Y|^{-(d+s)}\,\int_0^\infty\,\frac{w^{(d+s)/2-1}\,e^{-w}\,dw}{\prod_{\alpha>0}\,(|X-Y|^2/(4\,w)+\alpha(X)\,\alpha(Y))^{k}}\\
&\asymp t\,|X-Y|^{-(d+s)}\,\int_0^\infty\,\frac{w^{(d+s)/2+k\,m-1}\,e^{-w}\,dw}{\prod_{\alpha>0}\,(|X-Y|^2+\alpha(X)\,\alpha(Y)\,w)^{k}}\\
&\asymp \frac{t\,|X-Y|^{-(d+s)}}{\prod_{\alpha>0}\,(|X-Y|^2+\alpha(X)\,\alpha(Y))^{k}}
\end{align*}
with an application of Lemma \ref{ai} to get the last equivalence.  This proves the upper bound in the case $|X-Y|^2\geq t^{2/s}$.

We use Theorem \ref{heat} and the inequality \eqref{14}
again with the change of variable $u=t^{2/s}\,w^{-2/s}$.  Let $\Lambda$ be the set of $\alpha>0$ such that $t^{2/s}\leq\alpha(X)\,\alpha(Y)$
with $m'$ the number of elements in $\Lambda$. We have
\begin{align*}
h^W_t(X,Y)
&\lesssim t\,\int_0^\infty\,\frac{u^{-d/2}\,u^{-1-s/2}\,e^{-t\,u^{-s/2}}\,du}{\prod_{\alpha>0}\,(u+\alpha(X)\,\alpha(Y))^{k}}\\
&\asymp
t^{-d/s}\,\int_0^\infty\,\frac{w^{d/s}\,e^{-w}\,dw}{\prod_{\alpha>0}\,(t^{2/s}\,w^{-2/s}+\alpha(X)\,\alpha(Y))^{k}}
\\
&\asymp
t^{-d/s}\,\int_0^\infty\,\frac{w^{d/s+2\,k\,m/s}\,e^{-w}\,dw}{\prod_{\alpha>0}\,(t^{2/s}+\alpha(X)\,\alpha(Y)\,w^{2/s})^{k}}\\
&\lesssim
t^{-d/s}\,\frac{1}{(\prod_{\alpha\not\in \Lambda}\,t^{2/s})^{k}}
\,\int_0^\infty\,\frac{w^{d/s+2\,k\,m/s}\,e^{-w}\,dw}{\prod_{\alpha\in \Lambda}\,(\alpha(X)\,\alpha(Y)\,w^{2/s})^{k}}\\
&\asymp
\frac{t^{-d/s}}{\prod_{\alpha>0}\,(t^{2/s}+\alpha(X)\,\alpha(Y))^{k}}
\,\int_0^\infty\,w^{d/s+2\,k(m-m')/s}\,e^{-w}\,dw.
\end{align*}
This proves the upper bound in the case $|X-Y|^2\leq t^{2/s}$.

Now we will justify the lower bound. Recall (see \cite[(9,10), page 167]{Bogdan} or \cite[(9), page 89]{PG}) that for $u\geq t^{2/s}$, we have
\begin{equation}\label{asympeta}
\eta_t(u)\asymp t\,u^{-1-s/2}.
\end{equation}

If $|X-Y|^2\geq t^{2/s}$ then using formula \eqref{asympeta},
Theorem \ref{heat} and the change of variable $u=|X-Y|^2/(4\,w)$, we have
\begin{align*}
h^W_t(X,Y)&\gtrsim t\,\int_{t^{2/s}}^\infty\,\frac{u^{-d/2}\,e^{-|X-Y|^2/(4\,u)}\,u^{-1-s/2}\,du}{\prod_{\alpha>0}\,(u+\alpha(X)\,\alpha(Y))^{k}}\\
&\gtrsim t\,|X-Y|^{-(d+s)}\,\int_0^{|X-Y|^2/(4\,t^{2/s})}\,\frac{w^{(d+s)/2-1}\,e^{-w}\,dw}{\prod_{\alpha>0}\,(|X-Y|^2/(4\,w)+\alpha(X)\,\alpha(Y))^{k}}\\
&\gtrsim t\,|X-Y|^{-(d+s)}\,\int_{1/8}^{1/4}\,\frac{w^{(d+s)/2-1}\,e^{-w}\,dw}{\prod_{\alpha>0}\,(|X-Y|^2+\alpha(X)\,\alpha(Y))^{k}}
\end{align*}
which proves the lower bound in that case.

Now assume $|X-Y|^2\leq t^{2/s}$.  We use formula \eqref{asympeta} and
Theorem \ref{heat}. Then, since $-1/4\leq -|X-Y|^2/(4\,t^{2/s})\leq -|X-Y|^2/(4\,u)\leq 0$ for $u\geq t^{2/s}$, using the change of variable $u=t^{2/s}\,w$, we have
\begin{align*}
h^W_t(X,Y)&\gtrsim t\,\int_{t^{2/s}}^\infty\,\frac{u^{-d/2}\,e^{-|X-Y|^2/(4\,u)}\,u^{-1-s/2}\,du}{\prod_{\alpha>0}\,(u+\alpha(X)\,\alpha(Y))^{k}}\\
&\gtrsim t\,\int_{t^{2/s}}^\infty\,\frac{u^{-(d+s)/2-1}\,du}{\prod_{\alpha>0}\,(u+\alpha(X)\,\alpha(Y))^{k}}\\
&\gtrsim  t^{-d/s}\,\int_1^\infty\,\frac{w^{-(d+s)/2-1}\,dw}{\prod_{\alpha>0}\,(t^{2/s}\,w+\alpha(X)\,\alpha(Y))^{k}}\\
&\gtrsim  t^{-d/s}\,\int_1^\infty\,\frac{w^{-(d+s)/2-1}\,dw}{w^{k\,m}\,\prod_{\alpha>0}\,(t^{2/s}+\alpha(X)\,\alpha(Y))^{k}}\\
&\asymp  \frac{t^{-d/s}}{\prod_{\alpha>0}\,(t^{2/s}+\alpha(X)\,\alpha(Y))^{k}}
\end{align*}
which proves the lower bound in that case.
\end{proof}

\begin{remark}
The upper estimate  in Theorem \ref{th:stable} may be deduced from 
our  estimates of the $W$-invariant Dunkl heat kernel in Theorem \ref{heat}  and from \cite[Corollary 3.8 p. 11]{Trojan}. 
\end{remark}

\section*{Acknowledgments}
The authors wish to acknowledge the financial support from the grants IEA00292 CNRS 2021/22 and  MIR l'Universit\'e d'Angers ``Sym\'etries'' 2021.

The first author thanks the Dominican University College of Ottawa for their hospitality. The second author gratefully acknowledges the hospitality of l'Universit\'e d'Angers for his stay during which part of this paper was written.

The authors are thankful to Tomasz Grzywny, Tomasz Luks and Bartosz Trojan for helpful discussions and suggestions concerning  fractional Dunkl Laplacian and stable semigroups.


\begin{thebibliography}{99}
\bibitem{AnkerDH}
J.-P.{} Anker, J. Dziuba\'nski, A. Hejna.
\textit{
Harmonic Functions,Conjugate Harmonic Functions and
the Hardy Space $H^1$ in the Rational Dunkl Setting}, 
Journal of Fourier Analysis and Applications (2019) 25:2356-2418.

\bibitem{Bertoin} J.{} Bertoin.
\textit{L\'evy Processes}, Cambridge Tracts in Mathematics 1998.

\bibitem{BlGetoor}
 R. M. Blumenthal and R. K.. Getoor,
 \textit{Markov processes and potential theory}, Monographs in Pure and Applied Mathematics, Academic Press, New York, 1968.

\bibitem{Bogdan} K.{} Bogdan, A.{} St\'os and P.{}  Sztonyk P. \textit{Harnack inequality for stable processes on d-sets}, Studia Math, 158(2) (2003), 163--198.


\bibitem{Bou}
F. Bouzeffour, W. Jedidi, \textit{On the fractional Dunkl Laplacian} (2021),
 	arXiv:2110.08224 .


\bibitem{Gallardo2} L.{} Gallardo and C.{} Rejeb, \textit{Newtonian potentials and subharmonic functions associated to root systems}, Potential Anal. 47(4) (2017), 369--400.

\bibitem{Getoor} R. K. Getoor, \textit{
Infinitely divisible probabilities on the hyperbolic plane}, 
Pacific J. Math. 11(1961),  1287-1308.

\bibitem{Trojan} T.{} Grzywny and B.{} Trojan. \textit{Subordinated Markov processes: sharp estimates for heat kernels and Green functions}, arXiv preprint arXiv:2110.01201, 2021.

\bibitem{PGMRMY} P.{} Graczyk, M.{} R\"osler and M.{} Yor, Eds. \textit{Harmonic and stochastic analysis of Dunkl processes}, Hermann, 2008.

\bibitem{PGTLPS} P. Graczyk, T. Luks, and P. Sawyer, 
\textit{Potential kernels for radial Dunkl
Laplacians},  Canad.{} J.{} Math.{} (2019): 1--29.


\bibitem{PGPS1} 
P.{} Graczyk and P.{} Sawyer.
\textit{Sharp Estimates of Radial Dunkl and Heat Kernels in the Complex Case $A_n$}, Comptes Rendus Math\'ematique, Volume 359, issue 4 (2021), 427--437.


\bibitem{PG} P.{} Graczyk and A.{} St\'os. \textit{Transition density estimates for stable processes on symmetric spaces}, Pacific Journal of Mathematics, 217(1) (2004), 87--100.


\bibitem{Rejeb}
 Rejeb, C. \textit{Some results related to the fractional Dunkl Laplacian}(2021),  hal-03376564.


\bibitem{Roesler} M.{} R\"osler. \textit{Generalized Hermite polynomials and the heat equation for Dunkl operators}, Comm.{} Math.{} Phys.{} (1998), 519--542.

\bibitem{RoeslerHAB} M.{} R\"osler. \textit{
Contributions to the theory of Dunkl operators}. Habilitation Thesis, TU Munich, 1999.

\bibitem{RoeslerVoit} M.{} R\"osler, M. Voit, 
\textit{Markov Processes Related with Dunkl Operators},
Advances in Applied Mathematics
 21(1998), 575-643.


\bibitem{Sawyer} P.{} Sawyer, \textit{A Laplace-Type Representation of the Generalized Spherical Functions Associated with the Root Systems of Type $A$}, Mediterranean Journal of Mathematics, vol.{} 14, no.{} 4, 1--17, 2017.


\end{thebibliography}
\end{document}